\documentclass[12pt,a4paper,oneside]{article}
\usepackage{amsmath,amsthm,amssymb}
\usepackage[marginparwidth=1in]{geometry}
\usepackage{hyperref}

\usepackage{mdframed}
\usepackage{framed}

\usepackage{listings}
\usepackage{xcolor}
\definecolor{codegreen}{rgb}{0,0.6,0}
\definecolor{codegray}{rgb}{0.5,0.5,0.5}
\definecolor{codepurple}{rgb}{0.58,0,0.82}
\definecolor{backcolour}{rgb}{0.95,0.95,0.92}

\lstdefinestyle{mystyle}{
    backgroundcolor=\color{backcolour},   
    commentstyle=\color{codegreen},
    keywordstyle=\color{magenta},
    numberstyle=\tiny\color{codegray},
    stringstyle=\color{codepurple},
    basicstyle=\ttfamily\footnotesize,
    breakatwhitespace=false,         
    breaklines=true,                 
    captionpos=b,                    
    keepspaces=true,                 
    numbers=left,                    
    numbersep=5pt,                  
    showspaces=false,                
    showstringspaces=false,
    showtabs=false,                  
    tabsize=2
}
\usepackage[textsize=footnotesize,obeyFinal]{todonotes}
\usepackage{comment}

\newtheorem{theorem}{Theorem}
\newtheorem{lemma}[theorem]{Lemma}
\newtheorem{corollary}[theorem]{Corollary}
\newtheorem{conjecture}{Conjecture}
\newtheorem{remark}{Remark}

\newtheorem{question}{Question}

\newtheorem*{theorem*}{Theorem}

\newcommand{\A}{\mathcal{A}}
\newcommand{\minsuff}{\Gamma}

\begin{document}

\title{There exist infinite cube-free words over any sequence of binary alphabets}

\date{}
\author{Vuong Bui\thanks{Swinburne Vietnam, FPT University, Hanoi, 80 Duy Tan Street, Hanoi 100000, Vietnam (\texttt{bui.vuong@yandex.ru})} \and Matthieu Rosenfeld\thanks{LIRMM, Universit\'e de Montpellier, CNRS, 161 Rue Ada, 34095, Montpellier, France (\texttt{matthieu.rosenfeld@gmail.com})\\
 The second author was funded, in whole or in part, by the French National Research Agency (ANR) under grant agreement No. ANR-24-CE48-3758-01. 
}}

\maketitle

\begin{abstract}
    We prove that for any sequence of binary alphabets $\A_1,\A_2,\dots$, there exists a cube-free word $c_1c_2\dots$ so that $c_1\in\A_1,c_2\in\A_2,\dots$. In particular, for every $n$, there are at least $1.35^n$ cube-free words in $\A_1\times\A_2\times\dots\times \A_n$. We also prove that if the list of alphabets is computable then one of these words is computable and its $n$th letter can be computed in time polynomial in $n$.
\end{abstract}

\section{Introduction}
A \emph{square} (\textit{resp.} a \emph{cube}) is a word $uu$ (\textit{resp.} $uuu$) for some nonempty word $u$. A \emph{cube-free word} (resp. \emph{square-free word}) is a word that contains no cube (resp. square) as a factor.
Thue proved that there exists an infinite cube-free word over the binary alphabet and an infinite square-free word over the ternary alphabet \cite{Thue1}. These results are regarded as the first results in combinatorics on words, and many generalizations of these problems have been considered.

One such generalization is the notion of nonrepetitive coloring of graphs introduced by Alon et al. \cite{Alon2002}. A graph coloring (of the edges or of the vertices) is said to be \emph{nonrepetitive} if the sequence of colors along any path avoids squares (see \cite{WoodThueChoiceNumber} for a recent survey on this topic). This notion naturally led to the notion of nonrepetitive list-coloring, where instead of having one fixed set of colors every vertex has to choose a color from a list of colors specific to this vertex. The \emph{Thue-list number} of a graph is the smallest integer $k$ such that if all the lists have size at least $k$ then the graph can be nonrepetitively colored in such a way that each vertex receives a color from its list. In this article, we will consider a variant of the following challenging conjecture.
\begin{conjecture}[{\cite[\ldots]{rosenfeldThueList, CzerwinskiThueChoiceNumber, GagolThueChoiceNumber, GrytczukSurveyThueChoiceNumber, GrytczukThueChoiceNumber, WoodThueChoiceNumber,ZhaoThueChoiceNumber,Grytczuk2011Jan,Fiorenzi2011Oct,Skrabulakova2015Aug}}]\label{conj:ThueListe}
  The Thue list number of the infinite path is $3$.
\end{conjecture}

We provide a formulation of this conjecture in terms of combinatorics on words. By an \emph{alphabet}, we mean any finite set. Given a sequence $(\A_i)_{i\ge1}$ of alphabets, we say that a word $w=w_1\ldots w_n$ \emph{respects} $(\A_i)_{i\ge 1}$, if for all $i$, $w_i\in\A_i$ (the definition naturally extends to infinite words). Conjecture \ref{conj:ThueListe} can be rephrased as follows.
\begin{conjecture}
Given a sequence of alphabets $(\A_i)_{i\ge1}$ with  $|A_i|\ge3$ for all $i$, there exists an infinite square-free word that respects $(\A_i)_{i\ge 1}$.
\end{conjecture}

Without loss of generality, we can assume that all the $\A_i$ are sets of integers, so we adopt this convention for the sake of notation in the rest of the article. It was first proven by Grytczuk, Przyby{\l}o, and Zhu that the conjecture is true if the condition $|\A_i|\ge3$ is replaced by $|\A_i|\ge4$ \cite{Grytczuk2011Jan}.
The second author recently proved that if $|\bigcup_{i\ge1} \A_i|=4$ and $|\A_i|\ge3$ for every $i$ then there exists an infinite square-free word that respects $(\A_i)_{i\ge 1}$ \cite{rosenfeld2022avoiding}. This result is far from proving the conjecture, but the author argues that if it were not for the computational limitations, the technique used in this could certainly prove the full conjecture. Indeed, this technique requires one to verify by a computer-assisted proof some growth property of a very large automaton associated with square-free words.

In the current article, we consider a natural variant of the question. Since cubes are avoidable over the binary alphabet, is it possible to avoid cubes if all of the lists are of size $2$? We give a positive answer to this question in the following theorem. 

\begin{theorem} \label{thm:main}
    For any sequence of binary alphabet $\A_1,\A_2,\dots$, there exists an infinite cube-free word in $\A_1\times\A_2\times\dots$. 
\end{theorem}
In fact, we do not prove this result directly, but we prove two different stronger results, each implying Theorem \ref{thm:main}. 
These two proofs share a key idea: we can use an automaton to study the variant of the problem where we only forbid cubes of length at most $p$, for some arbitrary choice of $p$. Intuitively, short cubes are more difficult to avoid than long cubes, so by taking $p$ large enough, we should obtain a good approximation of the problem. Adapting standard tools from combinatorics on words and automata theory, studying this approximation is quite simple to achieve.
The delicate part is then to prove that if $p$ is large enough this approximation is good enough to deduce something about the original problem where all cubes are forbidden. 

The idea behind the first proof is to show that the growth rate of the language is not far from the growth rate of the regular approximation of the language which yields the following result.
\begin{theorem} \label{thm:growth}  
For any sequence of binary alphabet $\A_1,\A_2,\dots$, and for every $n$, there exist at least $1.35^n$ cube-free words in $\A_1\times\A_2\times\dots\times\A_n$. 
\end{theorem}
The core idea allowing us to deduce bounds on the growth of a power-free language $\mathcal L$ from its approximation $\mathcal L$ was first used by Kolpakov to estimate the growth rate of power-free languages \cite{kopakov2007efficient,Kolpakov2007Dec}. Shur later improved upon this technique \cite{shur2010Jul,Shur2012Nov}. This technique was later adapted by the second author to provide a partial answer to Conjecture \ref{conj:ThueListe} in  \cite{rosenfeld2022avoiding}.
The proof of this result applies the technique developed in \cite{rosenfeld2022avoiding} for squares, but provides a complete answer in the case of cubes.
The second proof of our main result does not imply anything about the number of solutions, but it contains algorithmic implications.
\begin{theorem} \label{thm:computable}
Let $\A_1,\A_2,\dots$ be a computable sequence of binary alphabet. Then there exists a computable infinite cube-free word $w=w_1w_2w_3\ldots$ in $\A_1\times\A_2\times\dots$. Moreover, there is an algorithm that given $w_1\ldots w_n$ and\ $\A_{n+1}$ computes $w_{n+1}$ in time polynomial in $n$.
\end{theorem}
This proof also relies on computing properties of the same automaton, but then we use a different argument for adding the large cubes back.
Intuitively, the idea is to introduce a weight function that measures how difficult it is to extend a word, and to prove that under some condition if the weight is not too large then at least one of the extensions also has a small weight (which we can then use inductively to construct a word). This idea was originally due to Miller \cite{miller2012two}, and was recently generalized in \cite{Rosenfeld2025Mar}. However, the weight functions in \cite{miller2012two} and \cite{Rosenfeld2025Mar} are much simpler and in our article the precise definition of the function depends on the automaton. Our result is in fact even stronger, as everything still holds if for all $i$ an opponent is allowed to choose the alphabet $\A_{i+1}$ after we choose $w_i$.

Note that Theorem \ref{thm:growth} and Theorem \ref{thm:computable} both independently imply Theorem \ref{thm:main}, but none of them seem to imply the other. We discuss the possible connection between this result in Section \ref{sec:discussion}.

The article is organized as follows. In Section \ref{sec:notation}, we introduce some common notations. In particular, we define the language $\widetilde{\mathcal{C}}^{(p)}$ that approximate the language of cube-free words, and we give the statement of Lemma \ref{lem:eigenvector} that is later crucial to our two proofs. In Section \ref{sec:potential}, we prove Theorem \ref{thm:computable}. In Section \ref{sec:growth}, we prove Theorem \ref{thm:growth}. Finally, we provide in Section \ref{sec:computations} a short description of the computation used to verify Lemma \ref{lem:eigenvector}.

\section{Approximation by a regular language}\label{sec:notation}
Our approach is to bound the difference between the number of cube-free words and the number of words that avoids all cubes up to a given length $p$. 
We let $\mathcal{C}$ be the set of all cube-free words. 
For a given integer $p$, we let $\widetilde{\mathcal{C}}^{(p)}$ be the set of words that avoid cubes of period at most $p$.
Given a sequence of alphabets $(\A_i)_{i\ge1}$, we let $\mathcal{C}[(\A_i)_{i\ge1}]$ (resp. $\widetilde{\mathcal{C}}^{(p)}[(\A_i)_{i\ge1}]$) be the elements of $\mathcal{C}$ (resp. $\widetilde{\mathcal{C}}^{(p)}$) that respect $(\A_i)_{i\ge1}$.

In order to work with $\widetilde{\mathcal{C}}^{(p)}$  (and $\widetilde{\mathcal{C}}^{(p)}[(\A_i)_{i\ge1}]$), we consider the (pseudo)-automaton that recognizes this language. The alphabet (and set of forbidden factors) is infinite, so this is not a regular language. But since the language $\widetilde{\mathcal{C}}^{(p)}$ is stable up to a permutation of the alphabet (and crucially, the set of forbidden factors up to permutation of the alphabet is finite), we can consider a (pseudo) finite automaton up to permutation of the alphabet. 
For this, one may construct the states as being all suffixes of length $3p-1$. 
However, we will instead adapt a state construction from Shur \cite{Shur2009}, which proves to be more efficient (both for the computer verifications and for the proofs), and was already used in similar settings where the alphabets are not fixed in \cite{rosenfeld2022avoiding, ROSENFELDAnnGame}. 

A \emph{minimal cube} is a cube that does not contain another cube as a proper factor.
A word is \emph{normalized} if it is lexicographically\footnote{To define a lexicographical order on the words, remember that we already assumed without loss of generality that
\(\A_i \subset\mathbb N \) for every $i$.} the smallest of all the words obtained by a permutation of the alphabet.
Informally, the states are given by the proper prefixes of all the minimal cubes with period at most $p$. More precisely, let $\minsuff$ be the set of normalized proper prefixes of minimal cubes of period at most $p$.
For any $w$, we abuse the notations and also let $\minsuff(w)$ be the longest word from $\minsuff$ that is a suffix of $w$ up to a permutation of the alphabet.

One easily verifies the following property of $\minsuff$.
\begin{lemma}\label{prop:minsuff}
For every word $u$ and every letter $a$, we have that 
\[
\minsuff(ua)=\minsuff(\minsuff(u)a)\,.
\]    
\end{lemma}
So we can think of $\minsuff$ as the set of states of the automaton\footnote{If we really wanted to make this an automaton we should add that the initial state is given by $\varepsilon$, every state is an accepting state. In fact, from every state there is at most one transition with any given label, so it is almost a deterministic automaton except that some edges are missing.} recognizing $\widetilde{\mathcal{C}}^{(p)}$ and for every state $s\in\minsuff$, the transition corresponding to the letter $a$ is simply given by $\minsuff(sa)$ (up to a permutation of the alphabet there are finitely many possible values for $a$: it is either one of the letter already appearing in $s$, or a new letter).

Using computer verifications that we detail in Section \ref{sec:computations}, we can obtain the following lemma.
\begin{lemma}\label{lem:eigenvector}
    Let $p=12$. There exist nonnegative weights $(\lambda_s)_{s\in \minsuff}$ with $\lambda_{\varepsilon}>0$ such that for every word $w$ over the alphabet $\mathbb{N}$
    and every $\A\in \binom{\mathbb{N}}{2}$,
    \begin{equation}\label{ineq:eigen}
        \alpha \lambda_{\minsuff(w)} \le \sum_{\substack{ c\in\A\\wc\in \widetilde{\mathcal{C}}^{(p)}}} \lambda_{\minsuff(wc)},
    \end{equation}
    where $\alpha = 1.457$.
\end{lemma}
It should be noted that this lemma is verified for the exact value $\alpha = 1.457$ which is a rational number and not with some floating point approximation of $\alpha$. Any reference to $p$ and $\alpha$ in the remainder of the article refers to these particular $p=12$ and $\alpha = 1.457$.

We now prove a second property of $\minsuff$ and $\mathcal{C}^{(p)}$ that we use in both proofs. First, we recall Fine and Wilf Theorem. An integer $p$ is said to be a \emph{period} of a word $u=u_1\ldots,u_k$ if for all $j\in\{1,\ldots, k-p\}$, we have  $u_j = u_{j+p}$ .
\begin{theorem}[Fine, Wilf]
If $i$ and $j$ are periods of a word $u$ and $|u|\ge i+j-\operatorname{gcd}(i,j)$, then $\operatorname{gcd}(i,j)$ is also a period of $u$.
\end{theorem}

\begin{lemma}
Suppose $v\in \mathcal{C}$ and $a$ is a letter such that
\begin{itemize}
    \item $va\in \widetilde{\mathcal{C}}^{(p)}$,
    \item $va = uyy$ for some words $u,y$ with $|y|>p$,
\end{itemize}
then $\minsuff(va)=\minsuff(yy)$.
\end{lemma}
\begin{proof}
For the sake of contradiction, assume $|\minsuff(va)|\ge |yy|$. By definition, 
$\minsuff(va)$ is the prefix of a cube of period at most $p$, so there exists some positive $i\le p$ such that $i$ is a period of $\minsuff(va)$, hence a period of $yy$. 
Both $i$ and $|y|$ are periods of $yy$ and $|yy|=2|y|>p+|y|\ge i+|y|>i+|y|-\operatorname{gcd}(|y|,i)$. By Fine and Wilf's theorem, $yy$ also has period $\operatorname{gcd}(|y|,i)<\frac{|y|}{2}= \frac{2|y|}{4}$. This implies that $yy$ (and $va$) contains a $4$-power, hence $v$ contains a cube, which contradicts $v\in \mathcal{C}$.
\end{proof}

\begin{remark}
This result can be improved by using sharper versions of the inequalities at the cost of complicating the argument.
In particular, here we partly exploit the existence of a cube and only use the existence of a square, which is a weaker condition.
It is not clear how the resulting improvement could simplify the computation. Although it is not critical for our computation with $p=12$, improving such lemmas could be useful for other applications of the same techniques (e.g. \cite{rosenfeldThueList}).
\end{remark}

In particular, we will use the following consequence in the proofs of Theorem \ref{thm:growth} and of Theorem \ref{thm:computable}.
\begin{corollary}\label{cor:easingOnePeriodPreservesState}
Suppose $v\in \mathcal{C}$, and $a$ is a letter such that
\begin{itemize}
    \item $va\in \widetilde{\mathcal{C}}^{(p)}$,
    \item $va = uyyy$ for some non-empty words $u,y$,
\end{itemize}
then $\minsuff(va)=\minsuff(yy)= \minsuff(uyy)$.
\end{corollary}

Corollary \ref{cor:easingOnePeriodPreservesState} and its proof can be seen as a particular case of a similar lemma used by Shur in the context of obtaining bounds for the growth of power-free languages \cite[Lemma 2]{Shur2009}.

\section{A polynomial-time algorithm to construct cube-free words}\label{sec:potential}

For any word $u\in \widetilde{\mathcal{C}}^{(p)}$, we define the weight $\omega(u)$ of $u$ as follows,
\[
    \omega(u)=
    \sum_{\substack{v\text{ such that: }\\uv\in \widetilde{\mathcal{C}}^{(p)}\\ uv=xyyy \text{ with }|y|\ge1, |v|<|y|}} \lambda_{\minsuff(uv)}\beta^{-|v|}\,,
\]
where $p=12$, $\alpha=1.457$ and the $\lambda_{\minsuff(uv)}$ are defined in Lemma \ref{lem:eigenvector}. This quantity is parameterized by $\beta$, and the particular choice of $\beta$ will be fixed later.
Note that $v$ can be $\varepsilon$ here.

Intuitively, this weight attempts to evaluate how difficult it will be to avoid cubes when extending $u$. For this, we consider cubes that are obtained by extending the word while remaining in $\widetilde{\mathcal{C}}^{(p)}$ (in some sense, it measures how likely it is that a long cube appears if we randomly follow the automaton). We only consider cubes such that at least two periods of the cube are contained in the suffix of $u$ (which will allow us to use Corollary \ref{cor:easingOnePeriodPreservesState}). The more letters we have to add to create a specific cube, the less weight we want to give to this cube since we will have more opportunities to avoid it, which motivates the $\beta^{-|v|}$ term. On the other hand, a word $u$ with a larger coefficient $\lambda_{\minsuff(u)}$ has more extensions in $\widetilde{\mathcal{C}}^{(p)}$, hence probably more extensions in $\mathcal{C}$ as well. 
Multiplying by this coefficient $\lambda_{\minsuff(u)}$ allows taking into account the fact that some words should be more costly to forbid than others because they have more extensions. Although we try to justify it by some intuitions, this definition is not simply a well-chosen heuristic, but a precise definition where every detail is crucial to the proof and can only be justified by the proof itself.

Let us state the first simple property of this weight function.
\begin{lemma}\label{lem:smallweighisfine}
Let $u\in \mathcal{C}$, and $a$ be a letter such that $ua\in \widetilde{\mathcal{C}}^{(p)}$. Then  $\omega(ua)< \lambda_{\minsuff(ua)}$ implies $ua\in \mathcal{C}$.
\end{lemma}
\begin{proof}
If $u\in \mathcal{C}$  and $ua\in \widetilde{\mathcal{C}}^{(p)}\setminus\mathcal{C}$, then $ua$ admits a cube as a suffix. Hence, considering the first term of the sum in the definition of $\omega(ua)$ (for $v=\varepsilon$) yields 
\[
\omega(ua)\ge  \lambda_{\minsuff(ua\varepsilon)}\beta^{-|\varepsilon|}
=\lambda_{\minsuff(ua)}\,
\]
as desired.
\end{proof}

Note that $\omega(\varepsilon)=0<\lambda_{\varepsilon}$ (due to Lemma \ref{lem:eigenvector}). The following lemma will be used to construct inductively a sequence $w$ where all prefix $u$  of $w$ is such that $u\in \widetilde{\mathcal{C}}^{(p)}$ and $\omega(u)< \lambda_{\minsuff(u)}$ so the whole sequence $w\in \mathcal{C}$ by the previous lemma.
\begin{lemma}\label{lem:betacond}
    Suppose $\beta>1$ is such that
    \[
    \beta+\frac{\beta^{1-p}}{\beta-1}\le\alpha.
    \]
    Let $u\in \mathcal{C}$ be such that $\omega(u)<\lambda_{\minsuff(u)}$, then for any alphabet $\A\subseteq\mathbb{N}$ of size at least $2$ there exists $a\in \A$ such that $ua\in \mathcal{C}$ and
    \[
     \omega(ua)<\lambda_{\minsuff(ua)}\,.
    \] 
\end{lemma}
\begin{proof}
Let $u$ and $\A$ be as in the theorem statement. By the definition of $\minsuff$, followed by Lemma \ref{prop:minsuff} and Lemma \ref{lem:eigenvector},
\begin{align*}
    \sum_{\substack{c\in\A\\uc\in \widetilde{\mathcal{C}}^{(p)}}} \lambda_{\minsuff(uc)}
    &= \sum_{\substack{c\in\A\\\minsuff(u)c\in \widetilde{\mathcal{C}}^{(p)}}} \lambda_{\minsuff(\minsuff(u)c)}\\
    &\ge \min_{\Sigma\in \binom{\mathbb{N}}{2}} \sum_{\substack{ c\in\Sigma\\\minsuff(u)c\in \widetilde{\mathcal{C}}^{(p)}}} \lambda_{\minsuff(\minsuff(u)c)}\\
    &\ge \alpha\lambda_{\minsuff(\minsuff(u))}=\alpha\lambda_{\minsuff(u)}\,.
\end{align*}

On the other hand, we consider the sum of the weights of the extensions of $u$ by a single letter: 
\begin{align*}
\sum_{\substack{c\in\A\\uc\in \widetilde{\mathcal{C}}^{(p)}}} \omega(uc)
&=\sum_{\substack{c\in\A\\uc\in \widetilde{\mathcal{C}}^{(p)}}}     
\sum_{\substack{v\text{ such that: }\\ucv\in\widetilde{\mathcal{C}}^{(p)}\\ ucv=xyyy \text{ with }|y|\ge1, |v|<|y|}} \lambda_{\minsuff(ucv)}\beta^{-|v|}\\
&\le \sum_{\substack{v'\text{ such that: }\\uv'\in\widetilde{\mathcal{C}}^{(p)}\\ uv'=xyyy \text{ with }|y|\ge1, |v'|\le|y|}} \lambda_{\minsuff(uv')}\beta^{1-|v'|}\,.
\end{align*}
This inequality is a direct consequence of the fact that if $cv$ is one of the summands in the first sum then it is a summand $v'=cv$ in the second sum. Splitting this last sum according to whether or not $|v|<|y|$, we obtain
\begin{align*}
\sum_{\substack{c\in\A\\uc\in \widetilde{\mathcal{C}}^{(p)}}} \omega(uc)&\le \beta\omega(u)+\sum_{\substack{v\text{ such that: }\\uv\in\widetilde{\mathcal{C}}^{(p)}\\ uv=xyyy \text{ with }|y|\ge1, |v|=|y|}} \lambda_{\minsuff(uv)}\beta^{1-|v|}.
\end{align*}

We have for all $u$,
\begin{align*}
&\{v : uv\in\widetilde{\mathcal{C}}^{(p)}, uv=xyyy \text{ with }|y|\ge1, |v|=|y|\}=\{ y: uy\in\widetilde{\mathcal{C}}^{(p)}, u = xyy, |y|>p\}\,.
\end{align*}
This set contains at most one element of each size larger than $p$ (and none for sizes at most $p$).
Moreover, by Corollary \ref{cor:easingOnePeriodPreservesState}, for all $y$ with $|y|>p$ such that $u=xyy$, we have $\lambda_{\minsuff(uy)}=\lambda_{\minsuff(u)}$. 

This implies 
\begin{align*}
    \sum_{\substack{c\in\A\\uc\in \widetilde{\mathcal{C}}^{(p)}}} \omega(uc)
    &\le \beta\omega(u)+\sum_{i>p} \lambda_{\minsuff(u)}\beta^{1-i}\\
    &<\lambda_{\minsuff(u)}\left(\beta+\frac{\beta^{1-p}}{\beta-1}\right)\\
    &\le\lambda_{\minsuff(u)}\alpha\\
    &\le  \sum_{\substack{c\in\A\\uc\in \widetilde{\mathcal{C}}^{(p)}}} \lambda_{\minsuff(uc)}\,,
\end{align*}
where the third inequality is by our theorem hypothesis, and the last inequality was proven earlier in this proof.
This implies that there exists $c\in \A$ such that $uc\in \widetilde{\mathcal{C}}^{(p)}$ and $\omega(uc)<\lambda_{\minsuff(uc)}$ as desired.
\end{proof}
\begin{proof}[Proof of Theorem \ref{thm:computable}]
  In particular, with $p=12$ and $\alpha=1.457$ we can set $\beta=1.35$ to satisfy the condition of Lemma \ref{lem:betacond}.
  So we inductively define $u_0=\varepsilon$, and for all $n$, $u_{n+1} = u_n a_{n+1}$ where $a_{n+1}\in\A_{n+1}$ is chosen such that $\omega(u_{n+1})<\lambda_{\minsuff(u_{n+1})}$. By  Lemma \ref{lem:smallweighisfine}, each of the $u_i$ is cube-free, and so is the infinite word $a_1a_2a_3\ldots$ and by construction it belongs to $\A_1\times\A_2\times\A_3\ldots$ as desired. 

  Given $u_n$ and $\A_{n+1}$, one only needs to compute the weight of $u_na$ for a given $a\in\A_{n+1}$ to choose $a_{n+1}$. For this, it is enough to be able to compute the set
  \[
    \{v | u_nv\in \widetilde{\mathcal{C}}^{(p)}, u_nv=xyyy \text{ with }|y|> p, |v|<|y|\}
  \]
  For any fixed choice of $|v|$ and $|y|$, $y$ is uniquely determined by $u_n$. Then so is $v$ and one can verify in time $O(|y|)<O(n)$ if this choice of $|v|$ and $|y|$ gives an element of this set (verifying if $u_nv\in \widetilde{\mathcal{C}}^{(p)}$, that is, if $yyy$ avoids cubes of period at most $p$ can be done in time linear in $|y|$). For the same reason, this set has size at most $O(n^2)$. We deduce that this set, and the weight $\omega(u_na)$ can be computed in time $O(n^3)$. So in particular $a_{n+1}$ can be computed in time $O(n^3)$ as desired.
\end{proof}

Note that the $O(n^3)$ bound is not optimal. In fact, given $|y|$ it takes $O(|y|)$ to find the smallest corresponding $|v|$ as it is equivalent to finding the largest suffix of $U_n$ of period $|y|$. It is certainly possible to improve this complexity, in particular, if we allow ourselves to keep useful information between every addition of a letter (we are more interested in generating a word of size $n$ than in being able to add one letter to a word of size $n$).

Our argument implies something slightly stronger than Theorem \ref{thm:computable}. Suppose that for every $n$, you can learn $\A_{n+1}$ only after having chosen $a_1,\dots,a_n$, then we can still apply the exact same technique. Indeed, in our proof when we choose the letter $a_n$, we never make any assumption on the following alphabets, which is consistent with the aforementioned setting (in particular, note that the weight $\omega$ does not depend on any of the alphabets, but only on the current word). Note that our second proof can also be adapted to say something about this setting, but it is not worth the extra technicalities.

In fact, the conclusion of Lemma \ref{lem:betacond} can be reformulated as saying that there is at most one letter $a\in\mathbb{N}$ such that $ \omega(ua)\ge\lambda_{\minsuff(ua)}$. This leads to an even stronger variant of the problem: instead of choosing the next letter an adversary provides an alphabet $\Sigma$, we forbid one letter $a\in\Sigma$ allowing the adversary to append any chosen letter from $\Sigma\setminus a$ to the current word. Our argument directly implies that it is possible for us to ensure that the resulting word avoid cubes for any strategy of the adversary. Furthermore, in this new setting, if $\Sigma$ is finite, the next move can always be computed.

\section{Cube-free words grow exponentially}\label{sec:growth}

    This section is devoted to proving Theorem \ref{thm:growth} on the exponential growth of the number of cube-free words of length $n$ that respect given alphabets $\mathcal A_1,\mathcal A_2,\dots$. For all $n\ge0$, let $C_n$ denote the set of words in $\mathcal C\left[(\mathcal A_i)_{i\ge 1}\right]$ of length $n$. We will show
    \[
        |C_n|\ge 1.35^n.
    \]
 For all $n\ge0$, let $C_n^t$ denote the set of cube-free words from $C_n$ that are in state $t$, that is,  
\[
C_n^t= \{w\in C_n: \minsuff(w)=t\}\,.
\]
We first prove an inequality about the size of these sets.

\begin{lemma}\label{lem:counting_argument}
For all $n\in\mathbb{N}$ and $t\in\minsuff$, we have the following inequality
    \begin{equation}\label{eq:exclusion}
        |C_{n+1}^t| \ge \sum_{v\in C_n}\sum_{\substack{c\in\A_{n+1}\\ vc\in\widetilde{\mathcal{C}}^{(p)}\\\minsuff(vc)=t}} 1 - \sum_{i=p+1}^{n+1} |C_{n+1-i}^t|.
    \end{equation}
\end{lemma}
\begin{proof}
    We let $G^t$ be the set of words of length $n+1$ in $A_1\times\ldots\times A_{n+1}$ so that
    \begin{itemize}
        \item the prefix of size $n$ is cube-free,
        \item the word does not contain any cubes of period at most $p$,
        \item the word ends in state $t$.
    \end{itemize} 
    In other words,
    \[
    G^t= \left\{va\in \widetilde{\mathcal{C}}^{(p)}: v\in C_{n}, a\in A_{n+1}, \minsuff(va)=t\right\}\,.
    \]
    Hence,
    \[
     |G^t|= \sum_{v\in C_n}\sum_{\substack{c\in\A_{n+1}\\ vc\in\widetilde{\mathcal{C}}^{(p)}\\\minsuff(vc)=t}} 1\,.
    \]
    Let $F_i$ be the set of words from $G_t$ that contain a cube of period $i$ as a suffix. By definition, a word from $G_t$ can only contain a cube as a suffix and this cube must be of period more than $p$ (and at most $n+1$).\footnote{We could easily replace this $n+1$ by $\lfloor(n+1)/3\rfloor$, but it does not make any difference in the latter part of the proof.} We get
    \[
    C_{n+1}^t= G^t \setminus \bigcup_{i=p+1}^{n+1}  F_i\,,
    \]
    which implies
    \[
    |C_{n+1}^t|\ge |G^t| - \sum_{i=p+1}^{n+1}  |F_i|\,.
    \]
    It remains to prove that $|F_i|\le |C_{n+1-i}^t|$ for every $i$ to conclude the proof.
    Let $va\in F_i$, then by definition $va=uyyy$ with $|y|=i>p$. Since $v\in C_n$ and $va\in \widetilde{\mathcal{C}}^{(p)}$, we can apply Corollary \ref{cor:easingOnePeriodPreservesState} to deduce that $\minsuff(uyy)=\minsuff(vc)=t$, that is, $uyy\in C_{n+1-i}^t$. Given the word $uyy$ and the period $i=|y|$, we can uniquely reconstruct $vc=uyyy$. This implies $|F_i|\le |C_{n+1-i}^t|$, which concludes our proof.
\end{proof}

Recall that, by Lemma \ref{lem:eigenvector}, we have nonnegative weights $\lambda_s$ for the states $s$ so that for every $s$ and every $\A\in\binom{\mathbb N}{2}$, we have
\[
        \alpha \lambda_s \le \sum_{\substack{ c\in\A\\sc\in \widetilde{\mathcal{C}}^{(p)}}} \lambda_{\minsuff(sc)}.
\]
Instead of considering the size of $C_n$ directly where every word is uniformly weighted, we use these coefficients $\lambda_s$ to assign different weights to the words in $C_n$, resulting in a better-behaved total weight of $C_n$.
Denote
\[
    \widehat{C_n}=\sum_{w\in C_n}\lambda_{\Gamma(w)}=\sum_{s\in\minsuff}\lambda_s |C_n^s|\,.
\]
We have the following result.

\begin{lemma}\label{prop:bijection}
For every $n$,
\[
	\widehat{C_{n+1}} \ge \alpha\widehat{C_n} - \sum_{i=p+1}^{n+1} \widehat{C_{n+1-i}}.
\]
\end{lemma}
\begin{proof}
It follows from Lemma \ref{lem:counting_argument} that
\begin{align*}
	\sum_t \lambda_t |C_{n+1}^t| &\ge\Bigl(\sum_t \lambda_t \sum_{v\in C_n}\sum_{\substack{c\in\A_{n+1}\\ vc\in\widetilde{\mathcal C}^{(p)}\\\minsuff(sc)=t}} 1\Bigr)- \sum_t \lambda_t \sum_{i=p+1}^{n+1} |C_{n+1-i}^t|\\
	&=\Bigl(\sum_{v\in C_n}\sum_t \sum_{\substack{c\in\A_{n+1}\\ vc\in\widetilde{\mathcal C}^{(p)}\\\minsuff(sc)=t}} \lambda_t\Bigr) - \sum_{i=p+1}^{n+1} \sum_t \lambda_t |C_{n+1-i}^t|\\
	&=\Bigl(\sum_{v\in C_n} \sum_{\substack{c\in\A_{n+1}\\ vc\in\widetilde{\mathcal C}^{(p)}}} \lambda_{\minsuff(vc)}\Bigr) - \sum_{i=p+1}^{n+1} \widehat{C_{n+1-i}}\\
	&\ge\sum_{v\in C_n} \alpha\lambda_{\minsuff(v)} - \sum_{i=p+1}^{n+1} \widehat{C_{n+1-i}} &\text{(due to Lemma \ref{lem:eigenvector})}\\
	&= \alpha\widehat{C_n} - \sum_{i=p+1}^{n+1} \widehat{C_{n+1-i}}.&\qedhere
\end{align*}
\end{proof}
We now prove the main ingredient in the proof of Theorem \ref{thm:growth}.
\begin{lemma}\label{lem:betacondBis}
If $\beta$ satisfies 
\[
    \alpha - \frac{\beta^{1-p}}{\beta-1} \ge \beta,
\]
then for every $n\ge 1$,
\[
    \widehat{C_n}\ge \beta \widehat{C_{n-1}}
\]
\end{lemma}

\begin{proof}
We prove it for every $n$ using strong induction method.
Assume the hypothesis holds up to some $n$, that is, $\widehat{C_i}\ge \beta\widehat{C_{i-1}}$ for every $i\le n$. We need to show that $\widehat{C_{n+1}}\ge \beta \widehat{C_n}$.
By the induction hypothesis, we have $\widehat{C_{n+1-i}} \le \beta^{1-i} \widehat{C_n}$, for all $i\ge1$. It follows from Lemma \ref{prop:bijection} that
\begin{align*}
	\widehat{C_{n+1}} &\ge \alpha\widehat{C_n} - \sum_{i=p+1}^{n+1} \widehat{C_{n+1-i}}\\
	&\ge \alpha\widehat{C_n} - \sum_{i=p+1}^{n+1} \beta^{1-i}\widehat{C_n}\\
	&\ge \alpha\widehat{C_n} - \sum_{i\ge p+1} \beta^{1-i}\widehat{C_n}\\
	&= \left(\alpha - \frac{\beta^{1-p}}{\beta-1}\right)\widehat{C_n}\\
    &\ge\beta \widehat{C_n}.
\end{align*}

We finish the induction step and conclude that $\widehat{C_n}\ge \beta \widehat{C_{n-1}}$ for every $n$.
\end{proof}

Since $\alpha=1.457$, we set
\[
    \beta=1.35
\]
to satisfy the condition $\alpha - \frac{\beta^{1-p}}{\beta-1} \ge \beta$.
It follows by induction that 
\[
    \widehat{C_n}\ge \beta^n \widehat{C_0}=\lambda_\epsilon \beta^n.
\]
Now we are ready to prove Theorem \ref{thm:growth}.
\begin{proof}[Proof of Theorem \ref{thm:growth}]
Denoting $\lambda^*=\max_s \lambda_s$, we have
\[
    |C_n|=\sum_s |C_n^s| = \frac{1}{\lambda^*} \sum_s\lambda^* |C_n^s| \ge \frac{1}{\lambda^*} \sum_s\lambda_s |C_n^s| = \frac{1}{\lambda^*} \widehat{C_n}.
\]
It follows that
\[
    |C_n|\ge \frac{1}{\lambda^*} \widehat{C_n} \ge \frac{\lambda_\epsilon}{\lambda^*} \beta^n.
\]
In other words, the number of cube-free words in $C_n$ is at least a constant times $\beta^n$. It follows that $\beta$ is a lower bound on the growth rate $\alpha$. 
On the other hand, we know that $|C_n|$ is submultiplicative, in the sense that $|C_{\ell+m}|\le |C_{\ell}||C_m|$ for every $\ell,m$. By Fekete's lemma, $|C_n|\ge\alpha^n$ for every $n$. In total,
\[
    |C_n|\ge\alpha^n\ge \beta^n = 1.35^n,
\]
which concludes Theorem \ref{thm:growth}.
\end{proof}

\section{Computer-assisted proof of Lemma \ref{lem:eigenvector}} \label{sec:computations} 
Following \cite{rosenfeld2022avoiding}, the way we establish $\alpha$ is inspired by the method of approximation of the spectral radius. We first explicitly compute the set of states $\minsuff$.  We initiate $\lambda_s=1$ for every state $s\in\minsuff$. In each iteration, we follow \eqref{ineq:eigen} in Lemma \ref{lem:eigenvector} to reassign $\lambda_s$ by
\begin{equation}\label{replacement}
        \lambda_s = \min_{\A\in \binom{[2+\max s]}{2}} \sum_{\substack{ c\in\A\\sc\in \widetilde{\mathcal{C}}^{(p)}}} \lambda_{\minsuff(sc)}.
\end{equation}
Note that the minimum is over $\binom{[2+\max s]}{2}$, instead of $\binom{\mathbb{N}}{2}$ as in \eqref{ineq:eigen}, but they are equivalent. Indeed, whenever the letters of $\A$ do not appear in $s$, we can always assume by symmetry that they are $|s|+1$ and $|s|+2$.
We repeat iterations until we have \eqref{ineq:eigen} with
\[
    \min\left\{ \frac{1}{\lambda_s}\min_{\A\in \binom{[2+\max s]}{2}}
    \sum_{\substack{ c\in\A\\sc\in \widetilde{\mathcal{C}}^{(p)}}} \lambda_{\minsuff(sc)}: s\in \minsuff\right\} \ge 1.457\,,
\]
which implies that \eqref{ineq:eigen} holds with the current values of $(\lambda_s)_{s\in\minsuff}$. It remains at this point to verify the condition $\lambda_\varepsilon>0$. All of the computation are done exactly using rational numbers which ensures that there is no problem of numerical approximations.
The program\footnote{Note that this program requires 20,2GB of RAM and a few minutes to run with the required parameters.} is available at the link \url{https://github.com/matthieu-rosenfeld/cube_free_over_list_of_size_2}

As mentioned, earlier, updating $\lambda_s$ with the result of \eqref{replacement} is really similar to the way one can find the spectral radius of a matrix. We verified in practice that this process converges for this particular system (at least, it seems to converge up to a really good approximation), but there should be a more general theoretical reason for the convergence.
\begin{question}
 Can we prove any guarantee on the convergence of this process?
\end{question}

\section{Potential connection between the two proofs}\label{sec:discussion}
As already mentioned in the introduction, Theorem \ref{thm:growth} and Theorem \ref{thm:computable} both independently imply Theorem \ref{thm:main}, but none of them seem to imply the other one. There are obvious connections between these proofs exemplified by the fact that in both of them we use the same computer verification provided by Lemma \ref{lem:eigenvector}, and we use Corollary \ref{cor:easingOnePeriodPreservesState} to bound the number of bad extensions. But there seem to be even deeper connections between these two results as in particular, both proofs lead to two lemmas (respectively, Lemma \ref{lem:betacond} and Lemma \ref{lem:betacondBis}) that require exactly the same algebraic condition, the existence of a real $\beta>1$ such that 
\[
    \beta+\frac{\beta^{1-p}}{\beta-1}\le\alpha.
\]
We wonder whether it is possible to better understand the similarities between these two proofs, to maybe obtain a shorter proof that directly implies both results.

\bibliographystyle{plain} 
\bibliography{cube-free}
\appendix
\end{document}